\documentclass[12pt,reqno]{amsart}   	
\usepackage{bm} 
\usepackage{amsfonts}
\usepackage{mathrsfs}
\usepackage{yfonts}
\usepackage{url}
\usepackage{pgfkeys}
\usepackage[osf]{mathpazo}  

\usepackage{mathtools}
\usepackage{amssymb}  		
\usepackage{graphicx}
\usepackage{verbatim}
\usepackage{enumerate}	
\usepackage[english]{babel}
\usepackage{amsmath}
\usepackage{bussproofs}	
\usepackage{quoting}
\usepackage{array,booktabs}

\usepackage[shortlabels]{enumitem}
\usepackage{units}
\usepackage{xspace}\xspaceaddexceptions{\,}
\usepackage{multicol}
\usepackage{booktabs,caption}
  \usepackage{tikz}
   \usepackage{tikz-cd}
\quotingsetup{font=small}	
\usepackage{amsthm}
\usepackage{amscd}

\usepackage[utf8]{inputenc}

\usepackage[pdftex,bookmarks,bookmarksnumbered,linktocpage,   
         colorlinks,linkcolor=blue,citecolor=blue]{hyperref}

\theoremstyle{definition}
\newtheorem{definition}{Definition}

\newtheorem{remark}[definition]{Remark}

\theoremstyle{plain}
\newtheorem{lemma}[definition]{Lemma}
\newtheorem{proposition}[definition]{Proposition}
\newtheorem{theorem}[definition]{Theorem}
\newtheorem{corollary}[definition]{Corollary}

\newcommand\A{{\mathbf A}}
\newcommand\B{{\mathbf B}}
\newcommand\C{{\mathbf C}}

\newcommand\I{{\mathbf I}}

\newcommand\WK{{\mathbf{WK}}}

\bmdefine{\boldstar}{\mathchoice{\textstyle*}{\textstyle*}{\textstyle*}{\scriptstyle*}}

\newcommand{\Jzero}{J_{_0}}
\newcommand{\Juno}{J_{_1}}
\newcommand{\Jdue}{J_{_2}}

\newcommand{\Ji}{J_{_k}}

\newcommand\Alg[1]{\if#1*\operatorname{\mathsf{Alg}*}\else\operatorname{\mathsf{Alg}}#1\fi}

\newcommand\Mod[1]{\if#1*\operatorname{\mathsf{Mod}*}\else\operatorname{\mathsf{Mod}}#1\fi}

\bmdefine{\Leibniz}{\Omega}        

\bmdefine{\frege}{\Lambda}         

\makeatletter
\newcommand{\tarskidsp}{\mathord%
   {\m@th\raisebox{0pt}[0pt][0pt]{$\stackrel%
   {\raisebox{-2.7pt}[0ex][0pt]{$\displaystyle \,\?\thicksim$}}%
   {\displaystyle\Leibniz}$}}}
\newcommand{\tarskitxt}{\mathord%
   {\m@th\raisebox{0pt}[0pt][0pt]{$\stackrel%
   {\raisebox{-2.7pt}[0ex][0pt]{$\,\?\thicksim$}}{\displaystyle\Leibniz}$}}}
\newcommand{\tarskiscr}{\mathord%
   {{\m@th\raisebox{0pt}[0pt][0pt]{$\stackrel%
   {\raisebox{-2.4pt}[0ex][0pt]{$\scriptstyle \,\?\thicksim$}}%
   {\scriptstyle\Leibniz}$}}}}
\newcommand{\tarskiscrscr}{\mathord%
   {{\m@th\raisebox{0pt}[0pt][0pt]{$\stackrel%
   {\raisebox{-2pt}[0ex][0pt]{$\scriptscriptstyle \,\?\thicksim$}}%
   {\scriptscriptstyle\Leibniz}$}}}}
\newcommand{\Tarski}{\@ifnextchar ^ %
   {\mathchoice{\tarskidsp\kern-.07em}{\tarskitxt\kern-.07em}%
   {\tarskiscr\kern-.07em}{\tarskiscrscr\kern-.07em}}%
   {\mathchoice{\tarskidsp}{\tarskitxt}{\tarskiscr}{\tarskiscrscr}}}
\makeatother
\DeclareMathAlphabet{\mathbfsf}{\encodingdefault}{\sfdefault}{bx}n
\makeatletter
\providecommand*{\Dashv}{\mathrel{\mathpalette\@Dashv\vDash}}
\newcommand*{\@Dashv}[2]{\reflectbox{$\m@th#1#2$}}
\makeatother
\newcommand\pair[1]{{\langle#1\rangle}}

\useshorthands{"}
\defineshorthand{"-}{{}\nobreakdash-\hspace{0pt}}	

\newcommand\PL{{\mathcal{P}}_{\text{\l}}}

\newcommand\ant{\nicefrac12}
\EnableBpAbbreviations

\newcommand{\bit}{\begin{itemize}}    
\newcommand{\eit}{\end{itemize}}
\newcommand{\ben}{\begin{enumerate}}
\newcommand{\een}{\end{enumerate}}

\newcommand{\bde}{\begin{description}}
\newcommand{\ede}{\end{description}}

\newcommand{\?}{\ensuremath{\mkern0.4\thinmuskip}}   

\newcommand\WKt{\WK^{\mathrm{e}}}
\newcommand{\alg}{\mathbf}
\newcommand{\class}{\mathsf}
\newcommand{\logic}{\textsc}

\begin{document}

\title{Bochvar algebras: A categorical equivalence and the generated variety}
\date{\today}

\author{Stefano Bonzio$^1$, Francesco Paoli$^2$ and Michele Pra Baldi$^3$}

\address{$^1$Department of Mathematics and Computer Science, University of Cagliari}

\address{$^2$Department of Pedagogy, Psychology and Philosophy, University of Cagliari}

\address{$^3$ FISPA Department, University of Padova}


\begin{abstract}
The proper quasivariety $\class{BCA}$ of Bochvar algebras, which serves as the equivalent algebraic semantics of Bochvar's external logic, was introduced by Finn and Grigolia in \cite{FinnGrigolia} and extensively studied in \cite{SMikBochvar}. In this paper, we show that the algebraic category of Bochvar algebras is equivalent to a category whose objects are pairs consisting of a Boolean algebra and a meet-subsemilattice (with unit) of the same. Furthermore, we provide an axiomatisation of the variety $V(\class{BCA})$ generated by Bochvar algebras. Finally, we axiomatise the join of Boolean algebras and semilattices within the lattice of subvarieties of $V(\class{BCA})$.
\end{abstract}

\maketitle

\section{Introduction}

In 1938, the Russian mathematician Dmitri Anatolyevich Bochvar published the influential paper \textquotedblleft On a three-valued logical calculus and its application to the analysis of the paradoxes of the classical extended functional calculus\textquotedblright \cite{Bochvar38}, in which he introduced a $3$-valued logic aimed at resolving set-theoretic and semantic paradoxes. His proposal diverged significantly from other related approaches in at least two key aspects. First and foremost, the third, non-classical value $\ant$ was \emph{infectious} in any sentential compound involving the standard, or \emph{internal}, propositional connectives $\lnot, \land, \lor$. This means that a formula would be assigned the value $\ant$ iff at least one variable within it was assigned $\ant$. This third value was interpreted as \textquotedblleft paradoxical\textquotedblright. Second, the language of Bochvar's logic included \emph{external} unary connectives $J_0,J_1,J_2$ (written in Finn and Grigolia's notation) that, unlike the internal connectives, could output only Boolean values.

Although the merits of Bochvar's logic as a solution to the paradoxes remain highly debatable\footnote{More promising applications of Bochvar's logic concern e.g. the treatment of fatal errors in computer science \cite{BergstraPonse1998, ferguson2017meaning} and the analysis of topicality in the philosophy of language \cite{beall2016off}.}, its influence on successive developments in $3$-valued logic has been significant. The internal fragment of Bochvar's logic was characterised by Urquhart \cite{Urquhart2001} through the imposition of a variable inclusion strainer on the consequence relation of classical propositional logic. Building on this result, a general framework for \emph{right variable inclusion logics} has been proposed (see \cite{Bonziobook} for a detailed account). In this context, the celebrated algebraic construction of \emph{P\l onka sums} is extended from algebras to logical matrices. Specifically, each logic  $\logic{L}$ is paired with a \textquotedblleft right variable inclusion companion\textquotedblright~$\logic{L}^r$ whose matrix models are decomposed as P\l onka sums of models of $\logic{L}$. Notably, Bochvar's internal logic serves as the right variable inclusion companion of classical logic.

Studies on Bochvar's external logic, by contrast, are comparatively scarce. Finn and Grigolia \cite{FinnGrigolia} provided an algebraic semantics for it with respect to the quasivariety of \emph{Bochvar algebras}. However, their work does not employ the standard toolbox or terminology of abstract algebraic logic. Adopting a more mainstream approach, the papers \cite{Ignoranzasevera, SMikBochvar} extend Finn and Grigolia's completeness theorem to a full-fledged algebraisability result, and offers a representation of Bochvar algebras that refines the P\l onka sum representations of their involutive bisemilattice reducts.\footnote{Bochvar's $3$-valued algebra was employed by Halld\'en \cite{Hallden} to introduce another logic that incorporates both internal and external connectives; crucially, however, Halld\'en -- unlike Bochvar -- treated $\ant$ as a designated value. The internal fragment of Halld\'en's logic has been thoroughly investigated in recent years under the name \emph{Paraconsistent Weak Kleene Logic} (see \cite[Chapter 7]{Bonziobook} for a survey). By contrast, the external logic has received relatively little attention in this case as well (exceptions include \cite{Segerberg65, BonzioZamperlin24}). It is important to observe that the quasivariety of Bochvar algebras algebraises not only Bochvar's, but also Halld\'en's external logic. Finally, it is worth mentioning that S.C. Kleene independently introduced the internal fragments of both logics in his \cite{KleeneBook}.}

In this paper, after covering the necessary preliminaries in Section \ref{prello}, we further refine the representation result in \cite{SMikBochvar}, proving that Bochvar algebras are categorically equivalent to \emph{Bochvar systems} -- pairs consisting of a Boolean algebra and a meet-subsemilattice with unit of such (Section \ref{prelly}). In Section \ref{prellu}, we axiomatise both the variety generated by Bochvar algebras, and its proper subvariety corresponding to the varietal join of Boolean algebras and semilattices (in the appropriate similarity type).

\section{Preliminaries}\label{prello}

We assume that the reader has some familiarity with universal algebra and  algebraic logic (standard references are \cite{BuSa00} and \cite{FontBook}, respectively). In what follows, given a class of similar algebras $\class{K}$, the usual class-operator symbols $I(\class{K}), S(\class{K}),H(\class{K}),P(\class{K}),P_{u}(\class{K})$ respectively denote the closure of $\class{K}$ under isomorphic copies, subalgebras, homomorphic images, products and ultraproducts. A class of similar algebras $	\class{K}$ is a quasivariety if $\class{K}=ISPP_{u}(K)$. When $K$ is a finite set of finite algebras, it holds that $ISP(\class{K})=ISPP_{u}(K)$. When $\class{K}$ is also closed under homomorphic images or, equivalently, if $\class{K}=HSP(\class{K})$, it is said to be a variety. $ V(\class{K})$ is shorthand for $HSP(\class{K})$.  Given a quasivariety $\class{Q}$, by  $\mathsf{Q}_{RSI}$ we indicate the class of relatively subdirectly irreducible members of $\mathsf{Q}$. If $\class{Q}$ is a variety, we simply write $\class{Q}_{SI}$. 

In this paper we will mainly deal with the algebraic language $\langle \land,\lor,\neg, \Jdue,0,1\rangle$, namely the expansion of the language of Boolean algebras via a unary operation symbol $\Jdue$. This is the language of Bochvar's external logic ($\logic{B}_{e}$) and external paraconsistent weak Kleene logic, the three-valued logics defined over the algebra $\alg{WK}^e$ displayed in Figure \ref{fig:WKe} by taking $\{1\}$ and $\{1,\ant\}$ as designated values, respectively.
\begin{figure}[h]
\begin{center}\renewcommand{\arraystretch}{1.25}
\begin{tabular}{>{$}c<{$}|>{$}c<{$}}
   & \lnot  \\[.2ex]
\hline
  1 & 0 \\
  \ant & \ant \\
  0 & 1 \\
\end{tabular}
\qquad
\begin{tabular}{>{$}c<{$}|>{$}c<{$}>{$}c<{$}>{$}c<{$}}
   \lor & 0 & \ant & 1 \\[.2ex]
 \hline
       0 & 0 & \ant & 1 \\
       \ant & \ant & \ant & \ant \\          
       1 & 1 & \ant & 1
\end{tabular}
\qquad
\begin{tabular}{>{$}c<{$}|>{$}c<{$}>{$}c<{$}>{$}c<{$}}
   \land & 0 & \ant & 1 \\[.2ex]
 \hline
     0 & 0 & \ant & 0 \\
     \ant & \ant & \ant & \ant \\          
    1 & 0 & \ant & 1
\end{tabular}

\vspace{10pt}
\begin{tabular}{>{$}c<{$}|>{$}c<{$}}
   & J_{_2} \\[.2ex]
\hline
  1 & 1 \\
  \ant & 0 \\
  0 & 0 \\
\end{tabular}

\end{center}
\caption{The algebra $\WKt$.}\label{fig:WKe}
\end{figure}

The quasivariety generated by $\alg{WK}^e$, that is, $ISP(\alg{WK}^e)$ is the equivalent algebraic semantics of the logics $\logic{B}^e$ and $\logic{PWK}^e$ (see \cite[Theorem 35]{Ignoranzasevera} and \cite[Theorem 7.1]{BonzioZamperlin24})
and is called the \emph{quasivariety of Bochvar algebras} -- $\class{BCA}$, in brief.
Two additional unary operations can be defined on $\alg{WK}^e$ from the operations $\Jdue$ and $\neg$, namely
$\Jzero a\coloneqq \Jdue \neg a$; $\Juno a\coloneqq \neg(\Jdue a\vee\Jdue\neg a)$. Those operations were taken as primitive in Finn and Grigolia's \cite[pp. 233-234]{FinnGrigolia}, where one finds the following quasi-equational basis for the quasivariety $\class{BCA}$.
\begin{definition}\label{def: algebre di Bochvar}
The quasivariety of Bochvar algebras can be axiomatised through the following identities and quasi-identities: 
\begin{enumerate}
\item $\varphi\vee \varphi\thickapprox \varphi$; \label{BCA:1}
\item$\varphi\lor \psi \thickapprox \psi \lor \varphi$; \label{BCA:2}
\item $(\varphi\lor \psi)\lor \delta\thickapprox  \varphi\lor(\psi\lor \delta)$; \label{BCA:3}
\item $\varphi\land(\psi\lor \delta)\thickapprox(\varphi\land \psi)\lor(\varphi\land \delta)$; \label{BCA:4}
\item $\neg(\neg \varphi)\thickapprox \varphi$; \label{BCA:5}
\item $\neg 1\thickapprox 0$; \label{BCA:6}
\item $\neg( \varphi\lor \psi)\thickapprox \neg \varphi\land\neg \psi$; \label{BCA:7}
\item $0\vee \varphi\thickapprox \varphi$; \label{BCA:8}

\item $\Jdue\Ji \varphi\thickapprox\Ji \varphi$, for every $k\in\{0,1,2\}$; \label{BCA:9}

\item $\Jzero\Ji \varphi\thickapprox \neg\Ji \varphi$, for every $k\in\{0,1,2\}$; \label{BCA:10}
\item $J_{_1}\Ji \varphi \thickapprox 0$, for every $k\in\{0,1,2\}$; \label{BCA:11}
\item $J_{_k}\neg \varphi\thickapprox J_{_{2-k}}\varphi$, for every $k\in\{0,1,2\}$; \label{BCA:12}
\item $J_{_i}\varphi \thickapprox\neg(J_{_j}\varphi\vee J_{_k}\varphi)$, for $i\neq j\neq k\neq i$; \label{BCA:13}

\item $J_{_k}\varphi\vee\neg J_{_k}\varphi \thickapprox 1$, for every $k\in\{0,1,2\}$; \label{BCA:14}

\item $(J_{_i}\varphi\vee J_{_k}\varphi)\land J_{_i}\varphi\thickapprox J_{_i}\varphi$, for $i,k\in\{0,1,2\}$; \label{BCA:15}
\item $ \varphi\lor J_{_k}\varphi \thickapprox \varphi$, for $k\in\{1,2\}$; \label{BCA:16}
\item $J_{_0}(\varphi\lor \psi)\thickapprox J_{_0}\varphi\land J_{_0}\psi$; \label{BCA:17}
\item $J_{_2}(\varphi\lor \psi)\thickapprox ( J_{_2}\varphi\land J_{_2}\psi)\vee( J_{_2}\varphi\land J_{_2}\neg \psi)\vee ( J_{_2}\neg \varphi\land J_{_2}\psi) $; \label{BCA:18}
\item $\Jzero \varphi \thickapprox \Jzero \psi \;\&\; \Juno \varphi \thickapprox \Juno \psi  \;\&\; \Jdue \varphi \thickapprox \Jdue \psi \;\Rightarrow\; \varphi \thickapprox \psi$. \label{BCA:quasi} 
\end{enumerate}
\end{definition}

The semantic properties of Bochvar algebras have been investigated in \cite{SMikBochvar}, where appropriate representation theorems are  provided and the following simplified axiomatisation for $\class{BCA}$ is offered:

\begin{theorem}\cite[Thm. 3.3]{SMikBochvar}\label{th: algebre di Bochvar2}
The following is a quasi-equational basis for $\class{BCA}$. 
\begin{enumerate}
\item $\varphi\vee \varphi\thickapprox \varphi$; \label{2BCA:1}
\item$\varphi\lor \psi \thickapprox \psi \lor \varphi$; \label{2BCA:2}
\item $(\varphi\lor \psi)\lor \delta\thickapprox  \varphi\lor(\psi\lor \delta)$; \label{2BCA:3}
\item $\varphi\land(\psi\lor \delta)\thickapprox(\varphi\land \psi)\lor(\varphi\land \delta)$; \label{2BCA:4}
\item $\neg(\neg \varphi)\thickapprox \varphi$; \label{2BCA:5}
\item $\neg 1\thickapprox 0$; \label{2BCA:6}
\item $\neg( \varphi\lor \psi)\thickapprox \neg \varphi\land\neg \psi$; \label{2BCA:7}
\item $0\vee \varphi\thickapprox \varphi$; \label{2BCA:8}
\item $\Jzero\Jdue \varphi\thickapprox \neg\Jdue \varphi$; \label{2BCA:10}
\item $\Jdue\varphi \thickapprox\neg(\Jzero\varphi\vee \Juno\varphi)$; \label{2BCA:13}
\item $\Jdue\varphi\vee\neg \Jdue\varphi \thickapprox 1$; \label{2BCA:14}
\item $J_{_2}(\varphi\lor \psi)\thickapprox ( J_{_2}\varphi\land J_{_2}\psi)\vee( J_{_2}\varphi\land J_{_2}\neg \psi)\vee ( J_{_2}\neg \varphi\land J_{_2}\psi) $; \label{2BCA:18}
\item $\Jzero \varphi \thickapprox \Jzero \psi \;\&\; \Jdue \varphi \thickapprox \Jdue \psi \;\Rightarrow\; \varphi \thickapprox \psi$.\label{2BCA:quasi}
\end{enumerate}
\end{theorem}

In the available representation theorems for $\class{BCA}$, the  key tool is a well-known algebraic construction introduced by Jerzy P\l onka in \cite{Plo67,Plo67a, Plonka1984nullary} (see also \cite{RomanowskaPlonka92, Bonziobook}) and later dubbed \emph{P\l onka sum}. A \emph{semilattice direct system} consists of a family of similar\footnote{The similarity type of such algebras must contain at least an operation symbol of arity $2$ or greater. This constraint will be kept implicit in what follows.} algebras $\{\A_{i}\}_{i\in I}$ with pairwise disjoint universes, such that the index set $I$ is the universe of a lower-bounded semilattice $\langle I, \vee, i_{0} \rangle$ with induced partial order $\leq$. Moreover, whenever $i\leq j$, for $i,j \in I$, we have a homomorphism $p_{ij}$ from the algebra $\A_i$ to the algebra $\A_j$. Such homomorphisms satisfy a \emph{compatibility property}: $p_{ii}$ is the identity, for every $i\in I$, and $p_{jk}\circ p_{ij} = p_{ik}$, for every $i\leq j\leq k$. The algebras in $\{\A_{i}\}_{i\in I}$ are called the \emph{fibres} of the system.

Given a semilattice direct system of algebras 
\[
\mathbb{A} = \pair{\{\A_i\}_{i\in I}, \mathbf{I} = \langle I, \vee, i_{0} \rangle, \{ p_{ij} : i \leq_\alg{I} j \}},
\] 
the \emph{P\l onka sum} over it is the algebra $\A= \PL(\mathbb{A}) = \PL(\A_{i})_{i\in I}$ (of the same similarity type as the algebras in $\{\A_i\}_{i\in I}$) whose universe is the disjoint union $A=\displaystyle\bigsqcup_{i\in I} A_i $ and where a generic $n$-ary term operation $g^\A$ is computed as follows: 
\begin{equation}\label{eq: operazioni}
g^{\A}(a_{1},\dots, a_n)\coloneqq g^{\A_k}(p_{i_{1}k}(a_1),\dots p_{i_{n}k}(a_n)),
\end{equation}
where $k= i_{1}\vee\dots\vee i_{n}$ and $a_1\in A_{i_1},\dots, a_{n}\in A_{i_n}$. If the similarity type contains any constant operation $e$, then $e^{\A} = e^{\A_{i_0}}$.  The \emph{fibres} of a P\l onka sum are the fibres of its underlying semilattice direct system. A fibre $\alg{A}_{i}$  is \emph{trivial} if its universe is a singleton. We will also say that a \emph{fixpoint} is the universe of a trivial fibre. Notice that if $a$ is a fixpoint, then $a=g( a)$ for every unary operation $g$.

For our purposes, the crucial application of P\l onka sums is with respect to semilattice direct systems of Boolean algebras. The resulting algebras have a remarkable feature: they satisfy all and only the valid Boolean identities $\varphi\thickapprox\psi$ that are \emph{regular}, i.e., such that the same variables appear in $\varphi$ and $\psi$ (this is a consequence of properties of the P\l onka sum construction, see e.g. \cite[Theorem 2.3.2]{Bonziobook}). 

The class of algebras obtained as  P\l onka sums over a semilattice direct system of Boolean algebras is a variety $\class{IBSL}$, whose members are known as \emph{involutive bisemilattices} \cite[Chapter 2]{Bonziobook}. Equivalently, this variety can be defined as follows \cite{Bonzio16SL}:

\begin{definition}\label{def: IBSL}
An \emph{involutive bisemilattice} is an algebra $\mathbf{B} = \pair{B,\land,\lor,\lnot,0,1}$ of type $\langle 2,2,1,0,0 \rangle$ satisfying:
\begin{enumerate}[label=\textbf{I\arabic*}.]
\item $\varphi\lor \varphi\thickapprox \varphi$;
\item $\varphi\lor\psi\thickapprox \psi\lor \varphi$;
\item $\varphi\lor(\psi\lor \delta)\thickapprox(\varphi\lor \psi)\lor \delta$;
\item $\lnot\lnot \varphi\thickapprox \varphi$;
\item $\varphi\land \psi\thickapprox\lnot(\lnot \varphi\lor\lnot \psi)$;
\item $\varphi\land(\lnot \varphi\lor \psi)\thickapprox \varphi\land \psi$; \label{rmp}
\item $0\lor \varphi\thickapprox \varphi$;
\item $1\thickapprox\lnot 0$.
\end{enumerate}
\end{definition}

Observe that all the identities in Definition \ref{def: IBSL} are regular. In the P\l onka sum representation  $\PL(\A_{i})_{i\in I}$ of an involutive bisemilattice $\A$, given $a,b\in A$, it is possible to check that $a,b\in A_{i}$ for some $i\in I$ if and only if $a\land(a\lor b)=a$ and $b\land(b\lor a)=b$. Moreover, if $a\in A_{i}$ and $i\leq j$, it holds $p_{ij}(a)=a\land(a\lor b)$, for any $b\in A_{j}$. 
 
We list some examples of involutive bisemilattices:
\begin{itemize}
    \item Boolean algebras, i.e., the involutive bisemilattices that satisfy one of the equivalent identities $\varphi \lor (\varphi \land \psi) \approx \varphi, \varphi \land \lnot \varphi \approx 0$ or $\varphi \lor \lnot \varphi \approx 1$; their P\l onka sum representations consist of a single fibre.
    \item Semilattices with zero, i.e., the involutive bisemilattices that satisfy one of the equivalent identities $\varphi \lor \psi \approx \varphi \land \psi$, $\lnot \varphi \approx \varphi$ or $0 \approx 1$; their P\l onka sum representations consist entirely of trivial fibres.
    \item The algebra $\alg{WK}$, namely, the $\Jdue$-free reduct of $\alg{WK}^{e}$ in Figure \ref{fig:WKe}. This algebra can be represented as the unique P\l onka sum whose fibres are a $2$-element Boolean algebra and a trivial one. $\alg{WK}$ generates the variety $\class{IBSL}$, in symbols $HSP(\alg{WK})=\class{IBSL}$.
\end{itemize}
In light of the above observations, it is easy to see that each Bochvar algebra has an $\class{IBSL}$-reduct, because the identities in Definition \ref{def: IBSL} coincide with the identities (1)-(8) of Theorem \ref{th: algebre di Bochvar2}. Given a Bochvar algebra $\A$, we will denote by $\PL(\A_{i})_{i\in I}$  the P\l onka decomposition of its $\mathsf{IBSL}$-reduct.

The subquasivariety  $\class{SIBSL}$ of $\class{IBSL}$, axiomatised relative to $\class{IBSL}$ by the quasi-identity 
\[
\varphi\thickapprox\neg\varphi \ \& \ \psi\thickapprox\neg\psi\Rightarrow \varphi\thickapprox\psi,
\]
comprises those members of $\class{IBSL}$ with at most one fixpoint, and is generated by $\alg{WK}$ as a quasivariety \cite[Thm. 7.1.19]{Bonziobook}.\footnote{This quasivariety is introduced  in \cite{Bonzio16SL}, while \cite{paoli2021extensions} considers the same class in a type with no constants.} Since each defining quasi-identity of $\class{SGIB}$ is valid in $\alg{WK}^e$, the following proposition holds.
\begin{proposition}\cite[Prop 3.5]{SMikBochvar}\label{prop: each BCA is SIBSL}
 Every Bochvar algebra has a $\mathsf{SIBSL}$-reduct.
 
\end{proposition}

 The connection between $\class{IBSL}$ and P\l onka sums of Boolean algebras provides some representation theorems for Bochvar algebras, as established in \cite{SMikBochvar}. We recall them below, as they play a key role in the subsequent sections. Notably, in the P\l onka sum representation of a Bochvar algebra all the homomorphisms are surjective. Consequently, each fibre is a quotient of the algebra $\A_{i_{0}}$; moreover, it is a quotient modulo a principal congruence.

{\begin{lemma}\label{lemma: omomorfismi sono suriettivi}
Let $\A$ be a Bochvar algebra with $\PL(\A_{i})_{i\in I}$ the P\l onka decomposition of its $\class{IBSL}$-reduct. Then 
\begin{enumerate}
\item $\PL(\A_{i})_{i\in I}$ has surjective homomorphisms;
\item  for every $i\neq i_{0}$, $p_{i_{0}i}$ is not injective;
\item for every $a\in A$ and every $i\in I$ (with $a\in A_i$), $\Jdue a\in p_{i_0 i}^{-1}(a)$ and $\Jzero a\in p_{i_0 i}^{-1}(\neg a)$.
\end{enumerate}
\end{lemma}}

\begin{theorem}\cite[Thm. 3.17]{SMikBochvar}\label{thm: converse decomposition}
Let $\A = \pair{A,\wedge,\vee,\neg, 0,1}$ be an involutive bisemilattice whose P\l onka sum representation is such that

\begin{enumerate}
 \item all homomorphisms are surjective and $p_{i_{0}i}$ is not injective for every $i_{0}\neq i\in I$;
 \item for each $i\in I$ there exists an element $a_{i}\in A_{i_{0}}$ such that $p_{i_{0} i}\colon \mathbf{[0},\mathbf{a}_{i}]\to \A_i$ is an isomorphism, with $a_{i}\neq a_{j}$ for $i\neq j$ and, in particular, $a_{j}<a_{i}$ for each $i<j$.
\end{enumerate}
Define, for every $a\in A_{i}$ and $i\in I$, $J_{_2}a=p_{i_{0}i}^{-1}(a)\in[0,a_{i}]$.
Then $\mathbf{B} = \pair{A,\wedge,\vee,\neg, 0,1, J_{_2}}$, 
is a Bochvar algebra. 
\end{theorem}

\begin{theorem}\cite[Cor. 3.18]{SMikBochvar}\label{excorollario21}
   Let $\A\in\class{SIBSL}$ be such that the underlying semilattice direct system of its P\l onka sum representation has surjective and non-injective homomorphisms.
 The following are equivalent:
\begin{enumerate}
 \item $\A$ is the reduct of a Bochvar algebra;
 \item for each $i\in I $, $1/ker_{p_{i_{0}i}}$ is a principal filter, with generator $a_{i}\in A_{i_{0}}$. Moreover,  if $i\neq j$  then $a_{i}\neq a_{j}$ and $a_{j}<a_{i}$ for each $i<j$;
 \item  for each $i\in I$ there exists an element $a_{i}\in A_{i_{0}}$ such that $p_{i_{0} i}\colon \mathbf{[0},\mathbf{a}_{i}]\to \A_i$ is an isomorphism. Moreover,  if $i\neq j$  then $a_{i}\neq a_{j}$ and $a_{j}<a_{i}$ for each $i<j$.
\end{enumerate}
\end{theorem}

Occasionally, if $\A$ is a Bochvar algebra and $\PL(\A_{i})_{i\in I}$ is the P\l onka decomposition of its $\class{IBSL}$-reduct, the top element of the fibre $\A_{i}$ will be denoted as $1^{A_i}$ or as $1_i$, while the bottom element of the same fibre will be denoted as $0^{A_i}$ or as $0_i$.

\section{A categorical equivalence for Bochvar algebras}\label{prelly}

Capitalising on the results that conclude the previous section, we aim to show that all the information about a Bochvar algebra $\A$, whose involutive bisemilattice reduct is represented as $\PL(\A_{i})_{i\in I}$, is encoded in two components: \emph{a)} the bottom fibre $\A_{i_{0}}$ of the attendant semilattice direct system, and \emph{b)} the meet-semilattice of the principal filters of $\A_{i_{0}}$ whose associated congruences determine the other fibres. Indeed, in this section we will establish that the algebraic category of Bochvar algebras is equivalent to a category whose objects are pairs consisting of a Boolean algebra $\B$ and a meet-subsemilattice with unit $\I$ of $\B$. We will observe that this category is a subcategory of a \emph{comma category} \cite{Lawvere} based on Boolean algebras and semilattices. This observation could open avenues for further exploration, particularly by linking Bochvar algebras to well-established algebraic and categorical concepts.

\begin{definition}
A \emph{Bochvar system} is a pair $\mathbb{B} = \langle \mathbf{B}, \mathbf{I}\rangle$ such that $\mathbf{B} = \langle B, \land, \lor, \lnot, 0, 1\rangle$ is a Boolean algebra and $\mathbf{I} = \langle I, \land, 1 \rangle$ is a meet-subsemilattice with unit of $\mathbf{B}$.
\end{definition}

Hereafter, for $b \in B$, by $[b)$ we denote the principal lattice filter generated by $b$ in $\B$, and for $F$ a filter of $\B$, by $\B/F$ we denote the quotient $\B/\theta_F$, where
\[
\theta_F = \{ \langle a, b \rangle \in B^2 : (\lnot a \lor b) \land (\lnot b \lor a) \in F\}.
\]
Our aim is to associate a Bochvar algebra to each Bochvar system. Thus, if $\mathbb{B} = \langle \mathbf{B}, \mathbf{I}\rangle$ is a Bochvar system, let
\[
\mathbb{A}_{\mathbb{B}} = \left\langle \{\mathbf{A}_{i}\}_{i\in I},\mathbf{I}^\partial,\{p_{ij}%
:i\leq_{\mathbf{I}^\partial}j\}\right\rangle
\]
be such that:
\begin{itemize}
    \item for all $i \in I$, $\mathbf{A}_i := \mathbf{B}/[i)$;
    \item $\mathbf{I}^\partial$ is the lower-bounded join-semilattice dual to $\mathbf{I}$;
    \item for all $i,j \in I$ such that $i\leq_{\mathbf{I}^\partial}j$, $p_{ij}(a/[i)) := (a/[i))/[j)$.
\end{itemize}

\begin{lemma}\label{somatizzo}
    $\mathbb{A}_{\mathbb{B}}$ is a semilattice direct system of Boolean algebras.
\end{lemma}

\begin{proof}
    It suffices to show that each $p_{ij}$ is well-defined, that $p_{ii}$ is the identity on $\mathbf{A}_i$ and that for $i \leq_{\mathbf{I}^\partial} j \leq_{\mathbf{I}^\partial} k$, $p_{jk} \circ p_{ij} = p_{ik}$.

    Since $\mathbf{I}$ is a subsemilattice with unit of $\mathbf{B}$, whenever $i\leq_{\mathbf{I}^\partial}j$, we have that $j \leq_{\mathbf{B}} i$ and hence $[i) \subseteq [j)$. Therefore $p_{ij}(a/[i)) = (a/[i))/[j)$ is a congruence class in $\mathbf{A}_i = \mathbf{B}/[i)$ and thus $p_{ij}$ is a well-defined surjective homomorphism from $\mathbf{A}_i$ to $\mathbf{A}_j$. By definition, $p_{ii}(a/[i)) = (a/[i))/[i) = a/[i)$ because the congruence that corresponds to the filter $[i)$ in $\B$ is the identity congruence in $\mathbf{B}/[i)$. Finally, by the General Homomorphism Theorem (see e.g. \cite[Thm. B2]{MPT}, 
    \[
    p_{jk} \circ p_{ij}(a/[i)) = ((a/[i))/[j))/[k) = (a/[i))/[k).
    \]
\end{proof}

Since $\mathbb{A}_{\mathbb{B}}$ is a semilattice direct system of Boolean algebras, $\PL(\mathbb{A}_{\mathbb{B}})$ is an involutive bisemilattice. We now prove that it can be expanded to a Bochvar algebra in a unique way.

\begin{theorem}\label{merendina}
    $\PL(\mathbb{A}_{\mathbb{B}})$ is the involutive bisemilattice reduct of a unique Bochvar algebra $\mathbf{A}_{\mathbb{B}}$.
\end{theorem}

\begin{proof}
    We use the equivalence between Conditions (1) and (2) in Theorem \ref{excorollario21}. First of all, we must show that $\PL(\mathbb{A}_{\mathbb{B}}) \in \class{SIBSL}$. As already remarked, $\PL(\mathbb{A}_{\mathbb{B}}) \in \class{IBSL}$; we must show that it has at most one trivial fibre. However, if $0^\mathbf{B} \in I$, $\mathbf{B}/[0^\mathbf{B})$ will be the unique trivial fibre in $\PL(\mathbb{A}_{\mathbb{B}})$; otherwise, $\PL(\mathbb{A}_{\mathbb{B}})$ will have no trivial fibre, because all the algebras in $\{\mathbf{A}_{i}\}_{i\in I}$ are quotients of $\B$ modulo congruences that are strictly less than $\nabla = B \times B$.

    Further, observe that for all $i \in I$,
    \begin{eqnarray*}
   ker_{p_{1i}} & = &\{ \langle a,b\rangle \in B^2 : p_{1i}(a) = p_{1i}(b) \} \\
    &  = & \{ \langle a,b\rangle \in B^2 : a/[i) = b/[i) \}
    \end{eqnarray*}
    so $1/ker_{p_{1i}} = \{ a \in B : \langle a,1\rangle \in ker_{p_{1i}}\} = \{ a \in B : a \in 1/[i) \} = [i)$, which is a principal filter of $\B$ with generator $i$. Moreover, if $i <_{\mathbf{I}^\partial} j$, then clearly $j <_{\mathbf{I}} i$ and thus $j <_{\mathbf{B}} i$, as $\mathbf{I}$ is a subsemilattice of $\mathbf{B}$.

    Uniqueness follows from the observation that if $a/[i) \in \A_i$, then $J_2(a/[i)) = p^{-1}_{1i}(a/[i))$, which is uniquely determined by Condition (3) in Theorem \ref{excorollario21}.
\end{proof}
After canonically associating a Bochvar algebra to a given Bochvar system, we now proceed in the opposite direction. Let 
\[
\A = \langle A, \land, \lor, \lnot, J_2, 0, 1\rangle
\]
be a Bochvar algebra, whose involutive bisemilattice reduct decomposes as $\PL(\A_i)_{i \in I}$. We define $\mathbb{B}_\A := \langle \A_{i_0}, \mathbf{K} \rangle$, where $K = \{ J_2^\A(1^{A_i}) : i \in I \}$, and for $J_2^\A(1^{A_i}),J_2^\A(1^{A_j}) \in K$, $J_2^\A(1^{A_i}) \leq_{\mathbf{K}} J_2^\A(1^{A_j})$ iff $j\leq_{\mathbf{I}} i$.

\begin{theorem}\label{bocciavera}
    $\mathbb{B}_\A$ is a Bochvar system.
\end{theorem}

\begin{proof}
    $\A_{i_0}$ is a Boolean algebra. $K$ is a subset of $B=A_{i_0}$, by Lemma \ref{lemma: omomorfismi sono suriettivi}. It contains $1^\mathbf{B} = 1^{\mathbf{A}_{i_0}} = J_2^\A(1^{\mathbf{A}_{i_0}})$. We now show that it is closed under meets in $\mathbf{B}$. Using Definition \ref{def: algebre di Bochvar}.(18) and the fact that for all $i \in I$ we have $J_2^\A(0^{A_i}) \leq^\A J_2^\A(1^{A_i})$,

    \begin{eqnarray*}
        J_2^\A(1^{A_i}) \land^{\B} J_2^\A(1^{A_j})&=&J_2^\A(1^{A_i}) \land^{\A_{i_0}} J_2^\A(1^{A_j}) \\
        &=&(J_2^\A(1^{A_i}) \land^{\A_{i_0}} J_2^\A(1^{A_j})) \\
        &&\lor^\A (J_2^\A(1^{A_i}) \land^{\A_{i_0}} J_2^\A(0^{A_j})) \\
        &&\lor^\A (J_2^\A(0^{A_i}) \land^{\A_{i_0}} J_2^\A(1^{A_j})) \\
        &=& J_2^\A(1^{A_i} \lor^\A 1^{A_j})) \\
        &=& J_2^\A(1^{A_{i \lor^\I j}}). \\
    \end{eqnarray*}
    
\end{proof}

The next goal in our agenda is showing that the previous correspondences between Bochvar algebras and Bochvar systems are mutually inverse (modulo isomorphism).

\begin{theorem}\label{montaldo}
    If $\A$ is a Bochvar algebra, then $\mathbf{A}_{\mathbb{B}_\A}$ is isomorphic to $\A$.
\end{theorem}

\begin{proof}
    Let $\A = \langle A, \land, \lor, \lnot, J_2, 0, 1\rangle$ be a Bochvar algebra, whose involutive bisemilattice reduct decomposes as $\PL(\A_i)_{i \in I}$. Recall that $\mathbb{B}_\A = \langle \A_{i_0}, \mathbf{K}\rangle$, where $K = \{ J_2^\A(1^{A_i}) : i \in I \}$. However, as a set of indices, we may as well use its bijective copy $I$, replacing when convenient each $J_2^\A(1^{A_i}) $ by $i$. Thus, taking into account the fact that the order of $\mathbf{K}$ dualises the order of $\mathbf{I}$, the P\l onka sum representation of the involutive bisemilattice reduct of $\mathbf{A}_{\mathbb{B}_\A}$ is
\[
\left\langle \{\mathbf{A}_{i}^*\}_{i\in I},\mathbf{I},\{p_{ij}^*%
:i\leq_{\mathbf{I}}j\}\right\rangle,
\]
where $\A_i^*= \A_{i_0}/[i)= \A_{i_0}/[J_2^\A(1^{A_i}))$ and for all $a \in A_{i_0},p_{ij}^*(a/[J_2^\A(1^{A_i})))=(a/[J_2^\A(1^{A_i})))/[J_2^\A(1^{A_j}))$.

To prove that $\mathbf{A}_{\mathbb{B}_\A}$ and $\A$ have isomorphic involutive bisemilattice reducts, it suffices to show that for all $i,j \in I$, $\mathbf{A}_{i}^* \cong \mathbf{A}_{i}$ and, up to isomorphism,
$p_{ij}^* = p_{ij}$. However, using Condition (3) in Theorem \ref{excorollario21}, $f(a) := J_2^\A a/[i)$ is an isomorphism between $\mathbf{A}_{i}$ and $\mathbf{A}_{i_0}/(1/ker_{p_{1i}})$ for every $i \in I$. So:
\[
\mathbf{A}_{i} \cong \mathbf{A}_{i_0}/(1/ker_{p_{1i}}) = \mathbf{A}_{i_0}/[i) = \mathbf{A}_{i}^*.
\]
Moreover, let $a \in A_i$. Identifying $a$ with $f(a)$,
\begin{eqnarray*}
    p_{ij}(J_2^\A a/[i))&=&(J_2^\A a/[i))/[j) \\
    &=&p_{ij}^*(J_2^\A a/[i)).
\end{eqnarray*} 
To round off our proof, it remains to be established that the term operations realised by $J_2$ in both algebras coincide. Indeed, applying Theorem \ref{thm: converse decomposition} and working again up to isomorphism:
\[
J_2^{\mathbf{A}_{\mathbb{B}_\A}} a=J_2^{\mathbf{A}_{\mathbb{B}_\A}}(J_2^\A a/[i)) = p_{i_0i}^{-1}(J_2^\A a/[i))=J_2^{\A}(J_2^\A a/[i))= J_2^{\A} a.
\]
\end{proof}
As implied by the definition of homomorphism to be found below, two Bochvar systems $\mathbb{B}_1 = \langle \mathbf{B}_1, \mathbf{I}_1\rangle$ and $\mathbb{B}_2 = \langle \mathbf{B}_2, \mathbf{I}_2\rangle$ are said to be isomorphic when there exists an isomorphism $g$ from $\mathbf{B}_1$ to $\mathbf{B}_2$ such that $g(i) \in I_2$ whenever $i \in I_1$.

\begin{theorem}\label{pinna}
If $\mathbb{B}= \langle \mathbf{B}, \mathbf{I}\rangle$ is a Bochvar system, then $\mathbb{B}$ is isomorphic to $\mathbb{B}_{\A_{\mathbb{B}}}$.
\end{theorem}

\begin{proof}
    Recall that to $\mathbb{B}= \langle \mathbf{B}, \mathbf{I}\rangle$ is associated a unique Bochvar algebra $\A_{\mathbb{B}}$ over the semilattice direct system of Boolean algebras
    \[
\mathbb{A}_{\mathbb{B}} = \left\langle \{\mathbf{A}_{i}\}_{i\in I},\mathbf{I}^\partial,\{p_{ij}%
:i\leq_{\mathbf{I}^\partial}j\}\right\rangle,
    \]
    defined before Lemma \ref{somatizzo}, so that $\mathbb{B}_{\A_{\mathbb{B}}} = \langle \A_{1^\B}, \mathbf{K}\rangle$, where $K = \{ J_2^{\A_{\mathbb{B}}}(1^{\A_{\mathbb{B}_i}}) : i \in I\}$ and $J_2^{\A_{\mathbb{B}}}(1^{\A_{\mathbb{B}_i}}) \leq_\mathbf{K} J_2^{\A_{\mathbb{B}}}(1^{\A_{\mathbb{B}_j}})$ iff $j \leq_{\mathbf{I}^\partial}i$ iff $i \leq_\mathbf{I} j$. 

    Firstly, we must show that $\B \cong \A_{1^\B}$, which is clear since $\A_{1^\B} = \B/[1^\B) \cong \B$. The other requirement is obvious. Actually, we even have $\mathbf{K} \cong \mathbf{I}$. Indeed, let $f: I \to K$ be defined by $f(i) = J_2^{\A_{\mathbb{B}}}(1^{\A_{\mathbb{B}_i}})$. $f$ is bijective by Theorem \ref{excorollario21}. Moreover,
    \begin{eqnarray*}
        f(i \land^\mathbf{I} j)&=&f(i \lor^{\mathbf{I}^\partial} j) \\
        &=&J_2^{\A_{\mathbb{B}}}(1^{\A_{\mathbb{B}_{i \lor j}}}) \\
        &=&J_2^{\A_{\mathbb{B}}}(1^{\A_{\mathbb{B}_{i}}}) \land^\mathbf{K} J_2^{\A_{\mathbb{B}}}(1^{\A_{\mathbb{B}_{j}}})\\
        &=&f(i) \land^\mathbf{K} f(j).
    \end{eqnarray*}
\end{proof}

In the following, $\mathfrak{B}$ will denote the algebraic category of Bochvar algebras. We now define a category $\mathfrak{S}$ whose objects are Bochvar systems. If $\mathbb{B}_1 = \langle \mathbf{B}_1, \mathbf{I}_1\rangle$ and $\mathbb{B}_2 = \langle \mathbf{B}_2, \mathbf{I}_2\rangle$ are objects in $\mathfrak{S}$, a morphism from $\mathbb{B}_1$ to $\mathbb{B}_2$ is a homomorphism $g$ from $\mathbf{B}_1$ to $\mathbf{B}_2$ such that $g(i) \in I_2$ for every $i \in I_1$. Observe that any such $g$ is also a homomorphism from $\mathbf{I}_1$ to $\mathbf{I}_2$. 

Let us briefly recall the notion of a \emph{comma category}. If $\mathfrak{A,B,C}$ are categories and $S,T$ are functors, respectively from $\mathfrak{A}$ to $\mathfrak{C}$ and from $\mathfrak{B}$ to $\mathfrak{C}$, called the \emph{source} and the \emph{target} functor, the comma category $S \downarrow T$ is defined as follows:
\begin{itemize}
    \item Its objects are triples $\langle A,B,f \rangle$, with $A$ an object in $\mathfrak{A}$, $B$ an object in $\mathfrak{B}$ and $f: S(A) \to T(B)$ a morphism in $\mathfrak{C}$.
    \item The morphisms from $\langle A,B,f \rangle$ to $\langle A',B',f' \rangle$ are pairs $\langle g,h \rangle$ where $g: A \to A'$ is a morphism in $\mathfrak{A}$, $h: B \to B'$ is a morphism in $\mathfrak{B}$ and the following diagram commutes:
    \begin{equation*}
\text{$\begin{CD} S(A)@>S(g) >> S(A')\\ @VV f V
@VVf'V\\ T(B) @>T(h) >> T(B')
\end{CD} $}
\end{equation*}

\end{itemize}

Now, consider the comma category $S \downarrow T$, where the source functor $S$ is the identity functor from the algebraic category $\mathfrak{Sem}$ of meet semilattices with unit to itself, and the target functor $T$ is the forgetful functor from the algebraic category $\mathfrak{Bool}$ of Boolean algebras to $\mathfrak{Sem}$. $\mathfrak{S}$ can be viewed as a full subcategory of $S \downarrow T$, whose objects are those $\langle A,B,f \rangle$ such that $f$ is injective (since $A$ must be a subreduct of $B$) and whose morphisms are the $S \downarrow T$-morphisms between such objects. Therefore, $\mathfrak{S}$ is well-defined.

We now define the map $\Gamma$ as follows:
\begin{itemize}
    \item If $\A$ is an object in $\mathfrak{B}$, let $\Gamma(\A) := \mathbb{B}_\A$.
    \item If $f: \A_1 \to \A_2$ is a morphism in $\mathfrak{B}$, let $\Gamma(f)$ be the restriction of $f$ to $\A_{1_{i_0}}$.
\end{itemize}
Similarly, we define the map $\Xi$ as follows:
\begin{itemize}
    \item If $\mathbb{B}$ is an object in $\mathfrak{S}$, let $\Xi(\mathbb{B}) := \mathbf{A}_\mathbb{B}$.
    \item If $g: \mathbb{B}_1 \to \mathbb{B}_2$ is a morphism in $\mathfrak{S}$, let $\Xi(g)$ be defined as follows: $\Xi(g)(a/[i)) := g(a)/[g(i))$.
\end{itemize}

\begin{lemma}\label{scateni}
    $\Gamma$ is a functor from $\mathfrak{B}$ to $\mathfrak{S}$.
\end{lemma}

\begin{proof}
    If $\A$ is an object in $\mathfrak{B}$, $\Gamma(\A)$ is an object in $\mathfrak{S}$ by Theorem \ref{bocciavera}. Next, we take care of morphisms. Let $\A_1,\A_2$ be Bochvar algebras (with the associated P\l onka representations of their involutive bisemilattice reducts), let $f$ be a homomorphism from $\A_1$ to $\A_2$, and let $a \in A_{1_{i_0}}$. Then
    \[
    \Gamma(f)(a) = f\restriction A_{1_{i_0}} (a) = f(a) \in A_2.
    \]
    However, if $a \in A_{1_{i_0}}$, $a \lor^{\A_1} \lnot^{\A_1} a = 1^{\A_1}$, and thus
    \[
    f(a) \lor^{\A_2} \lnot^{\A_2} f(a) = f(a \lor^{\A_1} \lnot^{\A_1} a) = f(1^{\A_1}) = 1^{\A_2}
    \]
    and hence $f(a) \in A_{2_{i_0}} $.
    Since $f$ is a homomorphism of Bochvar algebras, its restriction to $\A_{1_{i_0}}$ is a homomorphism of Boolean algebras. Moreover, for $i \in I_1$,
    \[
    \Gamma(f)(J_2^{\A_{1}}(1^{\A_{1_{i}}})) = f(J_2^{\A_{1}}(1^{\A_{1_{i}}})) = J_2^{\A_{2}}(f(1^{\A_{1_{i}}})) = J_2^{\A_{2}}(1^{\A_{2_{f^{\ast}(i)}}}),
    \]
    where $f^{\ast}$ is the homomorphism from $\I_1$ to $\I_2$ induced by $f$, viewed as a homomorphism of involutive bisemilattices.
\end{proof}

\begin{lemma}\label{fenu}
    $\Xi$ is a functor from $\mathfrak{S}$ to $\mathfrak{B}$.
\end{lemma}

\begin{proof}
  If $\mathbb{B}$ is an object in $\mathfrak{S}$, $\Xi(\mathbb{B})$ is an object in $\mathfrak{B}$ by Theorem \ref{merendina}. Let $\mathbb{B}_1 = \langle \mathbf{B}_1, \mathbf{I}_1\rangle,\mathbb{B}_2 = \langle \mathbf{B}_2, \mathbf{I}_2\rangle$ be Bochvar systems, let $g$ be a morphism from $\mathbb{B}_1$ to $\mathbb{B}_2$, and let $a/[i) \in A_{\mathbb{B}_1}$, where $a \in B_1$, $i \in I_1$. Then $g(a) \in B_2,g(i) \in I_2$, and thus $\Xi(g)(a/[i)) = g(a)/[g(i))$ is well-defined. We have to show that it is a homomorphism of Bochvar algebras. We check just the clauses for $\land$ and $J_2$. To keep the notation reasonably unencumbered, we will not add superscripts to the semilattice joins of $\I_1$ and $\I_2$, relying on the context to disambiguate.
  \begin{eqnarray*}
    \Xi(g)(a/[i) \land^{\mathbf{B}_1} b/[j))  &=&\Xi(g)(p_{i,i \lor j}(a/[i)) \land^{\mathbf{B}_{1_{i \lor j}}} p_{j,i \lor j}(b/[j)))\\
      &=&\Xi(g)((a/[i))/[i \lor j) \land^{\mathbf{B}_{1_{i \lor j}}} (b/[j))/[i \lor j))\\
      &=&\Xi(g)(a/[i \lor j) \land^{\mathbf{B}_{1_{i \lor j}}} b/[i \lor j))\\
      &=&\Xi(g)(a \land^{\mathbf{B}_{1}} b/[i \lor j) )\\
      &=&g(a \land^{\mathbf{B}_{1}} b)/[g(i \lor j))\\
      &=&(g(a) \land^{\mathbf{B}_{2}} g(b))/[g(i) \lor g(j))\\
      &=&g(a)/[g(i) \lor g(j)) \land^{\mathbf{B}_{2_{g(i) \lor g(j)}}} g(b)/[g(i) \lor g(j))\\
      &=&(g(a)/[g(i)) \land^{\mathbf{B}_{2}} g(b)/[g(j))\\
      &=&\Xi(g)(a/[i)) \land^{\mathbf{B}_{2}} \Xi(g)(b/[j)).
  \end{eqnarray*}
  \begin{eqnarray*}
    \Xi(g)(J_2^{\mathbf{B}_{1}}(a/[i)))  &=&\Xi(g)(p_{1,i}^{-1}(a/[i)))\\
      &=&\Xi(g)(a/\{ 1^{\mathbf{B}_{1}}\})\\
      &=&g(a)/[g(1^{\mathbf{B}_{1}}))\\
      &=&g(a)/\{ 1^{\mathbf{B}_{2}}\}\\
      &=&p_{1,g(i)}^{-1}(g(a)/[g(i)))\\
      &=&J_2^{\mathbf{B}_{2}}(g(a)/[g(i)))\\
      &=&J_2^{\mathbf{B}_{2}}(\Xi(g)(a/[i))).
  \end{eqnarray*}
\end{proof}

\begin{theorem}
    The functors $\Gamma$ and $\Xi$ induce a categorical equivalence between the categories $\mathfrak{B}$ and $\mathfrak{S}$.
\end{theorem}

\begin{proof}
    This results follows from Theorems \ref{montaldo} and \ref{pinna} and Lemmas \ref{scateni} and \ref{fenu} if we can prove that for any morphism $f:\A_1 \to \A_2$ in $\mathfrak{B}$ and any morphism $g:\mathbb{B}_1 \to \mathbb{B}_2$ in $\mathfrak{S}$, $\Xi(\Gamma(f)) = f$ and $\Gamma(\Xi(g)) = g$. We have to show the commutativity of the following diagram:
    \[
   \begin{tikzcd}
\A_1  \arrow{r}{f} \arrow{d}{\Xi\Gamma} & \A_2 \arrow{d}{\Xi\Gamma} \\
\Xi(\Gamma(\A_1)) \arrow{r}{\Xi\Gamma(f)} & \Xi(\Gamma(\A_2)) 
\end{tikzcd}
    \]
However, for $a \in \A_{1_i}$, 
\begin{eqnarray*}
  \Xi(\Gamma(f))(\Xi(\Gamma(a)))  &=&\Xi(\Gamma(f))(J_2 a/[i)) \\
    &=&f\restriction \A_{1_{i_0}} (J_2 a) / [f\restriction \A_{1_{i_0}}(i)) \\
    &=& J_2 f(a)/[f(i))\\
    &=&\Xi(\Gamma(f(a)). 
\end{eqnarray*}
The commutativity of the diagram
 \[
 \begin{tikzcd}
\mathbb{B}_1  \arrow{r}{g} \arrow{d}{\Gamma\Xi} & \mathbb{B}_2 \arrow{d}{\Gamma\Xi} \\
\Gamma(\Xi(\mathbb{B}_1)) \arrow{r}{\Gamma\Xi(g)} & \Gamma(\Xi(\mathbb{B}_2)) 
\end{tikzcd}
 \]
is shown similarly.
\end{proof}

\section{ On the variety generated by Bochvar algebras}\label{prellu}

In this section we investigate the variety generated by the proper quasivariety $\class{BCA}$. For a start, in Subsection \ref{acquacheta} we provide an equational basis for this variety and characterise the P\l onka sum representations of its members. Like $\class{IBSL}$, $V(\class{BCA})$ contains as subvarieties term-equivalent counterparts $\class{BA}$ of Boolean algebras and $\class{SL}$ of semilattices with zero -- in the former subvariety the operation symbol $J_2$ is interpreted as the identity function, while in the latter it is interpreted as the constant function yielding $0$ (or, equivalently, $1$) for each argument. In subsection \ref{gattamorta} we axiomatise the join of these (independent) varieties in the lattice of subvarieties of $V(\class{BCA})$, which coincides with the class of isomorphic copies of the product of a Boolean algebra and a semilattice.

\subsection{Axiomatisation and basic properties}\label{acquacheta}

\begin{definition}\label{def: equations VBCA}
$\class{K}$ is defined as the variety axiomatised by the following identities:
\begin{enumerate}[label=$\mathbf{K_{\arabic*}}$]
\item $\varphi\lor \varphi\thickapprox \varphi$; \label{K1}
\item $\varphi\lor\psi\thickapprox \psi\lor \varphi$; \label{K2}
\item $\varphi\lor(\psi\lor \delta)\thickapprox(\varphi\lor \psi)\lor \delta$; \label{K3}
\item $\lnot\lnot \varphi\thickapprox \varphi$; \label{K4}
\item $\varphi\land \psi\thickapprox\lnot(\lnot \varphi\lor\lnot \psi)$; \label{K5}
\item $\varphi\land(\lnot \varphi\lor \psi)\thickapprox \varphi\land \psi$; \label{K6}
\item $0\lor \varphi\thickapprox \varphi$; \label{K7}
\item $1\thickapprox\lnot 0$; \label{K8}
\item $\Jdue \varphi\lor\neg\Jdue\varphi\thickapprox 1$; \label{K9new}
\item $\varphi\lor J_{_{2}}\psi\thickapprox\varphi\lor J_{_{2}}(\varphi\lor\psi)$; \label{K10}
\item $\varphi\land J_{_{2}}\varphi\thickapprox \varphi$; \label{K11new}
\item $J_{_{2}}(\varphi\land\neg\varphi)\thickapprox 0$. \label{K12new}
\end{enumerate}
\end{definition}

We begin with a simple arithmetical lemma. 

 \begin{lemma}\label{lemma: aritmetica K}
 In $\class{K}$, the following holds:
 \begin{enumerate}
     \item $\varphi\lor J_{_{2}}\varphi\thickapprox \varphi$;
     \item $\varphi \thickapprox J_{_{2}}\varphi \lor (\varphi \land \lnot \varphi)$;
     \item $\Jdue \Jdue \varphi \thickapprox \Jdue \varphi$;
     \item $\varphi\lor \lnot J_{_{2}}\psi\thickapprox\varphi\lor \lnot J_{_{2}}(\varphi\lor\psi)$.
 \end{enumerate}
 \end{lemma}
 \begin{proof}
(1) Let $\alg{A}\in\class{K}$. Observe that $\Jdue 0 = 0$. Indeed, using \ref{K4}, \ref{K5}, \ref{K8}, and \ref{K12new}, $0 = \Jdue(0 \wedge\neg 0) = \Jdue (0\wedge 1) = \Jdue 0$. From this we have, for any $a\in A$, that $a = a\vee 0 = a\vee\Jdue 0 = a\vee \Jdue(0\vee a) = a\vee \Jdue a$ by \ref{K7}, \ref{K10}. 

(2) Using \ref{K11new}, Item (1) and involutive bisemilattice properties, we obtain $\Jdue a \lor (a \land \lnot a) = (\Jdue a \lor a) \land (\Jdue a \lor \lnot a) = a \land (\Jdue a \lor \lnot a) = (a \land \Jdue a) \lor (a \land \lnot a) = a \lor (a \land \lnot a) = a$.

(3) Using \ref{K9new}, Item (2) and involutive bisemilattice properties, $\Jdue a = \Jdue\Jdue a \lor (\Jdue a \land \lnot \Jdue a) = \Jdue\Jdue a$.

(4) By De Morgan laws and the dual of \ref{K6}  it holds that $\varphi\lor\neg(\varphi\lor\psi)=\varphi\lor\neg\psi$. This, together with \ref{K10} yields
\begin{align*}
\varphi\lor\neg\Jdue\psi=\varphi\lor\neg(\varphi\lor\Jdue\psi)=&\\
\varphi\lor\neg(\varphi\lor\Jdue(\varphi\lor\psi))=&\varphi\lor \neg J_{_{2}}(\varphi\lor\psi).
\end{align*}

 \end{proof}

Recall that by $\alg{WK}^{e}$ we denote the $3$-element Bochvar algebra of Figure \ref{fig:WKe}, whose $\class{IBSL}$-reduct is $\alg{WK}$.
\begin{lemma}\label{lem: auxiliary for V(BCA)}
 Let $\alg{A}\in\class{K}$. Then,
\begin{enumerate}[(i)]
 \item $\alg{A}$ has an $\class{IBSL}$-reduct;
 \item The P\l onka representation of $\alg{A}$ has surjective homomorphisms;
 \item If the lowest fibre $\alg{A}_{0}$ in the P\l onka representation of $\alg{A}$ is a $2$-element Boolean algebra, then there is a unique way to turn the $\class{IBSL}$-reduct of $\alg{A}$ into a $\class{K}$-algebra. This is done by defining the $J_{_{2}}$ operation as follows, for every $a\in A$:
\[
J_{_{2}}a= \begin{cases}\tag{J-def}\label{J def in proof}
1 \text{ if }  a=1_{i} \text { and } i\in I^{+}\\
0 \text{ otherwise},
\end{cases}
\]
\end{enumerate}
where $I^{+} = \{i\in I \;:\; |A_i| > 1 \}$.
\end{lemma}

\begin{proof}
 $(i)$ is obvious as \ref{K1}-\ref{K8} axiomatise $\class{IBSL}$ (see Definition \ref{def: IBSL}). \\
 \noindent
 $(ii)$. By $(i)$, $\alg{A}$ admits a P\l onka sum representation; let $\langle I,\leq \rangle$ be its underlying semilattice. For $i,j\in I$, suppose  $i\leq j$ and let $a\in A_{j}$. By \ref{K9new} we have that $J_{_{2}}a\in A_{0}$. 
  Moreover, Lemma \ref{lemma: aritmetica K}.(1) and \ref{K11new} entail $J_{_{2}}a\land(J_{_{2}}a\lor a)=a$, i.e. $p_{0j}(J_{_{2}}a)=a$ . By the composition property of P\l onka homomorphisms, $a=p_{0j}(J_{_{2}}a)=p_{ij}(p_{0i}(J_{_{2}}a))$, thus $a$ is the image under $p_{ij}$ of $p_{0i}(J_{_{2}}a)$. This shows the claim.
  
  \noindent
  $(iii)$. Let $\alg{A}$ be such that the universe of the bottom fibre is $A_{0} = \{0,1\}$. Since every homomorphism $p_{ij}$ in the P\l onka sum represetation of $\alg{A}$ is surjective, every fibre $\alg{A}_i$ is such that $|A_i| \leq 2$. If $|A_i| = 1$ and $a\in A_i$, then $a = a\wedge\neg a$ and $\Jdue a = \Jdue (a\wedge\neg a) = 0$, by \ref{K12new}. If $|A_i| = 2$, then $\Jdue 0_i = \Jdue (0_i \wedge 1_i) = 0$ (again by \ref{K12new}). On the other hand, suppose by contradiction that $\Jdue 1_i = 0$, hence, by \ref{K11new} $1_i = 1_i\land\Jdue 1_i = 1_i\land 0 = 1_i\land p_{0i}(0) = 1_i\land 0_i = 0_i $, in contradiction with the fact that $|A_i| = 2$, i.e. $0_i\neq 1_i$.
 \end{proof}

\begin{lemma}\label{lem: fibre triviali per congruenza}
Let $\alg{A}\in\class{K}$ be such that $\alg{A}_{0}$ is a $2$-element Boolean algebra. Assume $a\in A_{i},b\in A_{j}$ and $\alg{A}_{i},\alg{A}_{j}$ are non-trivial. Then $\alg{A}_{i\lor j}$ is non-trivial. 

\begin{proof}
 Assume $a\in A_{i},b\in A_{j},$ and that $\alg{A}_{i},\alg{A}_{j}$ are non-trivial. Suppose, towards a contradiction, that $\alg{A}_{i\lor j}$ is trivial. Then $a\lor b\in A_{i\lor j}$, which by \ref{J def in proof} in Lemma \ref{lem: auxiliary for V(BCA)} entails $\Jdue(a\lor b)=0=\Jdue (0_{i}\lor 1_{j})$. By \ref{K10}, however, we have that $0_{i}=0_{i}\lor\Jdue(0_{i}\lor 1_{j})=0_{i}\lor\Jdue 1_{i \lor j}=1_{i}$, a contradiction.
\end{proof}
\end{lemma}

\begin{theorem}\label{th: axiomatic base VBCA}
The variety generated by $\class{BCA}$ is $\class{K}$.
\end{theorem}
\begin{proof}
 Since $V(\mathsf{BCA})=H(\class{BCA})=HSP(\WKt)=V(\WKt)$ and \ref{K1}-\ref{K12new}  hold in $\WKt$, clearly $V(\class{BCA})\subseteq \class{K}$. We will obtain the converse inclusion by showing that $\class{K}_{SI}\subseteq HS(\WKt)$,  where $HS(\WKt)$ consists of the trivial algebra, the $2$-element semilattice, the $2$-element Boolean algebra and $\WKt$. Our strategy is an adaption of the results contained in \cite{Kal71}, and it proceeds as follows.
   Let  $\alg{A}\in\class{K}_{SI}$. By Lemma \ref{lem: auxiliary for V(BCA)}, $\alg{A}$ has an $\class{IBSL}$-reduct with surjective homomorphisms and underlying semilattice $I$. Let  $0$ be the least element of $I$.   For $a\in A_{0}$ define the relation 
  \[
  \Theta_{a}\coloneqq\{ \langle b,c \rangle \in\alg{A}^{2}: a\lor b= a\lor c, a\lor\neg b= a\lor\neg c \}.
  \]
  $\Theta_{a}$ is an $\class{IBSL}$ congruence \cite[Lm. 24]{Bonzio16SL}. Moreover, suppose $ \langle b,c \rangle \in\Theta_{a}$. Then:
  \[a\lor\Jdue b=a\lor\Jdue(a\lor b)=a\lor\Jdue(a\lor c)=a\lor\Jdue c,\]
  
\noindent where the first and the last equalities hold by \ref{K10}, while the second one holds because of the assumption $ \langle b,c \rangle \in\Theta_{a}$. Similarly, using Lemma \ref{lemma: aritmetica K}.(4), we obtain
   \[a\lor\neg\Jdue b=a\lor\neg\Jdue(a\lor b)=a\lor\neg\Jdue(a\lor c)=a\lor\neg\Jdue c,\]
  \noindent thus $\Theta_{a}$ is compatible with $\Jdue$ and it is a congruence on $\alg{A}$.

  We claim that $\Theta_{a}=\Delta$ if and only if $a=0$. The right-to-left direction of this claim is obvious. Conversely, assume $\Theta_{a}=\Delta$. For any $b\in A$ we have $a\lor b=a\lor(a\lor b)$ and 
  \[a\lor\neg(a\lor b)=a\lor(\neg a\land\neg b)=(a\lor\neg a)\land(a\lor\neg b)=a\lor\neg b,\]
  where the above equivalences are justified by the fact that De Morgan and Distributivity laws hold in $\class{IBSL}$.
 Hence, $b=a\lor b$ for every $b\in A$, which is true if and only if $a=0$. This proves the claim.
 
    We now show that $\alg{A}_{0}$ is either a trivial algebra or the $2$-element Boolean algebra. 
  For $a\in A_{0}$, let $\langle b,c \rangle \in\Theta_{a}\cap\Theta_{\neg a}$. From the definition of this congruence it follows that $a\lor b=a\lor c$ and $a\land b=a\land c$. Thus
\begin{align*}
 b\land(b\lor a)&=b\land(c\lor a)\\
 &=(b\land c)\lor (b\land a)\\
 &=(b\land c)\lor (c\land a)\\
 &=c\land(c\lor a).
\end{align*}
Since $a\in A_{0}$, for any $i\in I$ and $d\in A_{i}$ we have that $d\land(d\lor a)=d$. So, the above equivalences entail that $b=c$, namely $\Theta_{a}\cap\Theta_{\neg a}=\Delta$.
Since $\alg{A}$ is subdirectly irreducible, $\Theta_{a}=\Delta$ or $\Theta_{\neg a}=\Delta$ which by the above claim entails $a=0$ or $\neg a=0$ (i.e. $a=1$). Therefore, we have shown that $\alg{A}_{0}$ is either trivial or a $2$-element Boolean algebra. If it is a trivial algebra, then $\alg{A}$ is a semilattice by $(ii)$ in Lemma \ref{lem: auxiliary for V(BCA)}, thus $\alg{A}$ is the $2$-element semilattice, i.e., the unique subdirectly irreducible semilattice. 

So, suppose that $\alg{A}_{0}$ is a $2$-element Boolean algebra. Notice that for every $i\in I$, $\alg{A}_{i}$ is either trivial or a $2$-element Boolean algebra, because each homomorphism is surjective. Suppose, towards a contradiction, that $\alg{A}\notin HS(\alg{WK}^{e})$.  We distinguish two cases, depending on the presence of non-trivial fibres in the P\l onka representation of $\alg{A}$. 

Let us begin with the case in which there is at most one  trivial fibre. Our assumption that $\alg{A}\notin HS(\alg{WK}^{e})$, together with (\ref{J def in proof}),  ensures that $\alg{A}_{0}$ cannot be the only non-trivial fibre, for otherwise $\alg{A}\in HS(\alg{WK}^{e})$. Define a relation $\Theta^{\prime}$ as follows:

\[
\langle a,b \rangle \in \Theta^{\prime} \iff \begin{cases}
a=b \text{ or}  \\
 \exists i,j\in I^{+}, a=1_{i}, b=1_{j}\text{ or}\\
\exists i,j\in I^{+}, a=0_{i}, b=0_{j},
\end{cases}
\]
where $I^{+}=\{i\in I: \alg{A}_{i} \text{ is non trivial}\}.$

Furthermore, let  us now define $\Theta^{\star}$ as:
\[\langle a,b \rangle \in\Theta^{\star} \iff \begin{cases}
a=b \text{ or}  \\
 \exists i\in I \text{,} a,b\in A_{i}.
\end{cases}
\]
Clearly both are equivalence relations. We now show that  $\Theta^{\prime}$ is a  congruence on $\alg{A}$. The compatibility with $\neg$ is obvious once we recall that whenever $a,b\in A_{i}$ then also $\neg a,\neg b\in A_{i}$. Suppose $\langle a, b \rangle \in\Theta^{\prime}$ and $\langle c, d \rangle \in\Theta^{\prime}$. Without loss of generality we may assume  $a\neq b$ and $a=1_{i},b=1_{j},c=0_{k},d=0_{z}$.
 Therefore $\alg{A}_{i}, \alg{A}_{j}$ are non-trivial, by definition of $\Theta^{\prime}$. If one among $\alg{A}_{k}$ and $\alg{A}_{z}$ is trivial, then $c=d$ and $k=z$, which will be the top element of $I$. 
 Therefore, since  $\alg{A}$ has at most one trivial fibre, $a\lor c \Theta b\lor d$, as desired. Otherwise, $\alg{A}_{k},\alg{A}_{z}$ are non-trivial and $c\neq d$. By Lemma \ref{lem: fibre triviali per congruenza}, $\alg{A}_{i\lor k}$ and $\alg{A}_{j\lor z}$ are non trivial, thus $a\lor c=1_{i\lor k}$ and $b\lor d=1_{j\lor z}$, namely $\langle a\lor c,b\lor d \rangle \in\Theta^{\prime}$. Suppose now $\langle a,b \rangle \in \Theta^{\prime}$ so, without loss of generality $a=1_{i},b=1_{j}$ for some $i,j\in I^{+}$. By (\ref{J def in proof}) in Lemma \ref{lem: auxiliary for V(BCA)} we obtain $\Jdue a=1=\Jdue b$, thus $\Theta^{\prime}$ is compatible with $\Jdue$ and it is a congruence on $\alg{A}$.

 $\Theta^{\star}$ is an involutive bisemilattice congruence (see e.g. \cite[Thm. 2.3.11]{Bonziobook}) and $\Jdue a\in A_{0}$ for each $a\in A$, so $\Theta^{\star}$ is compatible with $\Jdue$.
Recall that, as already noticed, there are additional nontrivial fibres other than $\alg{A}_{0}$: let  $\alg{A}_{q}$ be one of them, distinct from $\alg{A}_{0}$. Thus, $\langle 1,1_{q} \rangle \in\Theta^{\prime}\neq\Delta\neq\Theta^{\star}$, because $\langle 0,1 \rangle \in \Theta^{\star}$. If we show $\Theta^{\prime}\cap\Theta^{\star}=\Delta$, we would then obtain the desired contradiction with the fact that $\alg{A}$ is  subdirectly irreducible. To this end, notice that for distinct elements $a,b$, we have that $\langle a,b \rangle \in \Theta\cap\Theta^{\star}$ if and only if either $a=1_{i},b=1_{j}$ and $i=j$ or  $a=0_{i},b=0_{j}$ and $i=j$. In both cases, this is true if and only if $a=b$, as desired.

We now consider the case in which there are two or more trivial fibres. Define the relation $\Theta$ as:
 \[\langle a,b \rangle \in\Theta \iff \begin{cases}
a=b \text{ or}  \\
a=\neg a, b=\neg b.
\end{cases}
\]
This is an equivalence relation. We show that it is a 
congruence on $\alg{A}$. If $\langle a,b \rangle \in\Theta$,  $a=\neg a, b=\neg b$, so $\langle \neg a, \neg b \rangle \in\Theta$. Moreover $\Jdue a=0=\Jdue b$, by $(iii)$ in Lemma \ref{lem: auxiliary for V(BCA)}, i.e. $\langle \Jdue a,\Jdue b\rangle \in\Theta$. If $\langle a,b \rangle \in\Theta, \langle c,d \rangle \in\Theta$, then $a\lor c, b\lor d$ must be fixpoints. This proves that $\Theta$ is a congruence. Notice that $\Theta^{\star}$ is a congruence on $\alg{A}$ also in this case, for the very same argument used previously. 
Observe also that $\Theta\neq\Delta\neq\Theta^{\star}$, because $\langle 0,1 \rangle \in\Theta^{\star}$, while $\Theta$ collapses (at least) the two trivial fibres.  We proceed to show that $\Theta\cap\Theta^{\star}=\Delta$. We have that $\langle a,b \rangle \in\Theta\cap\Theta^{\star}$ if and only if  $a=\neg a, b=\neg b$ and  $a,b\in A_{i}$ for some $i\in I$. However, this is true if and only if $a=b$, again a contradiction.

Therefore, we conclude $\alg{A}\in HS(\alg{WK}^{e})$.
 This proves $K\subseteq V(\class{BCA})$ and, finally, $\class{K}=V(\class{BCA})$.  
 \end{proof}

\begin{remark}\label{rem: forbidden conf}
Observe that Lemma \ref{lem: fibre triviali per congruenza}  allows us to build a $\class{SIBSL}$ with surjective homomorphisms that cannot be equipped with a $V(\class{BCA})$ structure. This is the case of the algebra displayed in Figure \ref{forb}, let us call it $\alg{A}$. This algebra is a $\class{SIBSL}$ whose underlying semilattice $I$ consists of four elements: the bottom element, $i,j$ and $k=i\lor j$. Notice that $i,j$ are incomparable elements in $I$. The corresponding fibres $\alg{A}_{0},\alg{A}_{i},\alg{A}_{j}$ are $2$-element Boolean algebras, while $\alg{A}_{k}$ is a trivial algebra. Dotted lines represent P\l onka homomorphisms. Suppose now that $\alg{A}$ can be equipped with an appropriate operation $\Jdue$ satisfying \ref{K1}-\ref{K12new}. By Lemma \ref{lem: auxiliary for V(BCA)}, we have that $\Jdue(0_{i})=0=\Jdue(1_{j}\lor 0_{i})=\Jdue(0_{k})$. By \ref{K10}, however, we have that $0_{i}=0_{i}\lor\Jdue(0_{i}\lor 1_{j})=0_{i}\lor\Jdue 1_{j}=1_{i}$, a contradiction. This shows that any $\class{IBSL}$ having $\alg{A}$ as a subalgebra cannot be turned into a $V(\class{BCA})$ algebra.
\end{remark}

\begin{figure}[h]\caption{A forbidden configuration}\label{forb}
 \begin{center}
\begin{tikzpicture}[scale=1, dot/.style={inner sep=2.5pt,outer sep=2.5pt}, solid/.style={circle,fill,inner sep=2pt,outer sep=2pt}, empty/.style={circle,draw,inner sep=2pt,outer sep=2pt}]

 \node  [label={above:$0_{k}=1_{k}$}](0k) at (4,7.5) [solid] {};

\node  [label={right:$0_{i}$}](0i) at (6,4) [solid] {};
  \node  [label={right:$1_{i}$}] (1i) at (6,5.5) [solid] {};

\node  [label={left:$0_{j}$}](0j) at (2,4) [solid] {};
 \node  [label={left:$1_{j}$}] (1j) at (2,5.5) [solid] {};

 \node  [label={above:$1$}](1) at (4,3.5) [solid] {};
 \node  [label={right:$0$}] (0) at (4,2) [solid] {};
 
  \draw[-] (0) edge (1);
 \draw[-] (0i) edge (1i);
  \draw[-] (0j) edge (1j);
    \draw[dotted] (0) edge (0j);
        \draw[dotted] (1) edge (1j);
  \draw[dotted] (0) edge (0i);
        \draw[dotted] (1) edge (1i);
        \draw[dotted] (1i) edge (0k);
        \draw[dotted] (0i) edge (0k);
\draw[dotted] (1j) edge (0k);
        \draw[dotted] (0j) edge (0k);
\end{tikzpicture}
\end{center}
\end{figure}
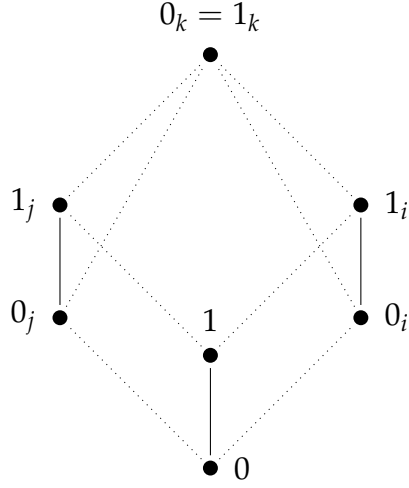



By Theorem \ref{th: axiomatic base VBCA}, we are in a position to help ourselves to all the equational properties of Bochvar algebras when carrying out computations concerning members of $\class{K}=V(\class{KBCA})$. For example, in what follows the axioms of Definition \ref{def: algebre di Bochvar} will be freely employed when necessary.

\begin{remark}\label{rem: V(BCA) non congruente distributivo} 
As previously noted, all semilattices with zero can be viewed as members of $V(\class{BCA})$, such that all primitive operation symbols other than $\lor$ and $0$ are either redundant, or realise the identity, or realise constant operations. For this reason $V(\class{BCA})$ does not satisfy any congruence identities in the type of lattices.
\end{remark}



\subsection{The varietal join of Boolean algebras and semilattices}\label{gattamorta}

In this subsection we show that $\class{BA}$ and $\class{SL}$, whose relative equational bases w.r.t. $V(\class{BCA})$ are given by the identities $\Jdue x \approx x$ and $\Jdue x \approx 1$ respectively, are \emph{independent} subvarieties of $V(\class{BCA})$ and, therefore, that their varietal join $\class{BA} \lor \class{SL}$ comprises precisely the isomorphic images of direct products of a Boolean algebra and a semilattice. Although this result directly provides an axiomatisation of $\class{BA} \lor \class{SL}$, using the algorithms given in, for instance, \cite{JonTsi} or  \cite{KLP}, we derive a much simpler equational basis by explicitly describing the direct product decomposition.

Recall that two subvarieties $\class{V}_{1},\class{V}_{2}$ of a variety $%
\class{V}$ of type $\tau $ are said to be \emph{independent} if there exists a binary term $\varphi (x,y)$ of type $\tau $ s.t. $\class{V}_{1}\vDash
\varphi (x,y) \approx x$ and $\class{V}_{2}\vDash \varphi (x,y) \approx y$. Recall, moreover, that the \emph{direct product} of two similar varieties $\class{V}_{1},\class{V}_{2}$ is defined as%
\begin{equation*}
\class{V}_{1}\times \class{V}_{2}:=I(\left\{ \mathbf{A}_{1}\times \mathbf{A}%
_{2} : \mathbf{A}_{1}\in \class{V}_{1},\mathbf{A}_{2}\in \class{V}_{2}\right\}) \text{.}
\end{equation*}

The following classical theorem about independent varieties is due to Gr\"{a}tzer, Lakser, and P\l onka \cite{GLP69}:

\begin{theorem}
\label{gratz}Let $\class{V}_{1},\class{V}_{2}$ be independent
subvarieties of an arbitrary variety $\class{V}$. Then $\class{V}_{1},%
\class{V}_{2}$ are such that their join $\class{V}_{1}\vee 
\class{V}_{2}$ (in the lattice of subvarieties of $\class{V}$) is their
direct product $\class{V}_{1}\times \class{V}_{2}$.
\end{theorem}

We aim at showing that the varieties $\class{BA}$ and $\class{SL}$ are independent. Before that, we confirm that it is legitimate to identify these classes with Boolean algebras and semilattices, respectively.

\begin{lemma}\label{terrapieno}
Let $\A \in V(\class{BCA})$. Then:
\begin{enumerate}
   \item $\A \models J_{2}x \approx x$ iff the involutive bisemilattice reduct of $\A$ is a Boolean algebra;
   \item $\A \models J_{2}x \approx 1$ iff the involutive bisemilattice reduct of $\A$ is a semilattice.
\end{enumerate}
\end{lemma}

\begin{proof}
    (1) If the involutive bisemilattice reduct of $\A$ is a Boolean algebra, its P\l onka sum representation has a single fibre, whence by \ref{K11new} and Lemma \ref{lemma: aritmetica K}.(1) for all $a \in A$ we have that $\Jdue a = \Jdue a \lor (\Jdue a \land a) = a$. Conversely, if $\A \models J_{2}x \approx x$, then for all $a \in A$ we have that $a \land \lnot a = J_{2}(a \land \lnot a) = 0$, hence the involutive bisemilattice reduct of $\A$ is a Boolean algebra.

    (2) If the involutive bisemilattice reduct of $\A$ is a semilattice, then $\A \models x \approx \lnot x$. Then by \ref{K9new} for all $a \in A$ we have that $1 = J_{2}a \lor \lnot J_{2}a = J_{2}a \lor J_{2}a = J_{2}a$. Conversely, if $\A \models J_{2}x \approx 1$, then for all $a \in A$ we have that $1 = J_{2}(a \land \lnot a) = 0$ by \ref{K12new}, which means that hence the involutive bisemilattice reduct of $\A$ is a semilattice with zero.
\end{proof}

\begin{theorem}\label{birignao}
    $\class{BA}$ and $\class{SL}$ are independent subvarieties of $V(\class{BCA})$.
\end{theorem}

\begin{proof}
    Using Lemma \ref{terrapieno}, it is easy to check that the term $\varphi(x,y) := J_{2}x \lor (J_{2}x \land y)$ witnesses independence for $\class{BA}$ and $\class{SL}$.
\end{proof}

\begin{corollary}\label{destrezza}
    $\class{BA} \lor \class{SL} = \class{BA} \times \class{SL}$.
\end{corollary}

\begin{proof}
    This follows from Theorem \ref{gratz} and Theorem \ref{birignao}.
\end{proof}

Given an algebra $\A$ in $V(\class{BCA})$, we single out two notable subsets of its universe: its \emph{open} elements and its \emph{dense} elements. In the P\l onka sum decomposition of its underlying involutive bisemilattice, the former correspond to the elements in the bottom fibre, while the latter correspond to the bottom elements in each fibre:
\begin{eqnarray*}
  O(\A):=\{ a \in A : J_2a=a\}; \\
  D(\A):=\{ a \in A : J_2a=0\} .
\end{eqnarray*}
\begin{lemma}\label{calzettoni}
Let $\A \in V(\class{BCA})$, and let $a \in A$.
    \begin{enumerate}
        \item $a \in O(\A)$ if and only if there exists $b \in A$ s.t. $a = J_2b$.
        \item $a \in D(\A)$ if and only if there exists $b \in A$ s.t. $a = b \land \lnot b$.
        \item $O(\A)$ is the universe of a Boolean subalgebra $\mathbf{O}(\A)$ of $\A$.
        \item $D(\A)$ is the universe of a semilattice $\mathbf{D}(\A)$ which is isomorphic to a quotient of $\A$.
    \end{enumerate}
\end{lemma}

\begin{proof}
    (1) The left-to-right direction is immediate. For the converse direction, by Lemma \ref{lemma: aritmetica K}.(3) we have that $a = J_2b = J_2J_2b = J_2a$.
    
    (2) If $a \in D(\A)$, then by Lemma \ref{lemma: aritmetica K}.(2) $a = J_2a \lor (a \land \lnot a) = 0 \lor (a \land \lnot a) = a \land \lnot a$. Conversely, suppose that $a = b \land \lnot b$. Then by \ref{K12new} $J_2a = J_2(b \land \lnot b) = 0$.

    (3) Clearly $J_21 = 1$. Let $J_2a = a$. Then using Definition \ref{def: algebre di Bochvar}.(10), $J_2\lnot a = J_0a = J_0J_2a = \lnot J_2a = \lnot a$. If $J_2a = a, J_2b = b$, then (having already proved that $O(\A)$ is closed under $\lnot$) by Definition \ref{def: algebre di Bochvar}.(18) we have that $J_2(a \lor b) = (J_2 a \land J_2 b) \lor (J_2 a \land J_2 \lnot b) \lor (J_2 \lnot a \land J_2 b) = (a \land b) \lor (a \land \lnot b) \lor (\lnot a \land b) = a \lor b$. Observe that the last equality holds because we are applying a $J_2$-free regular identity that is valid in all Boolean algebras. Finally, if  $J_2a = a$, then $J_2J_2a = J_2a$. Thus, $O(\A)$ is the universe of a subalgebra $\mathbf{O}(\A)$ of $\A$. This subalgebra is Boolean by Lemma \ref{terrapieno}.(1).

    (4) It is well-known (see e.g. \cite[Ch. 2]{Bonziobook}) that the relation $ \theta = \{\langle a,b\rangle : a \land \lnot a = b \land \lnot b \}$ is a congruence of the involutive bisemilattice reduct of $\A$. Moreover, if $a \theta b$, then $J_2a \theta J_2b$ since by \ref{K4}, \ref{K5}, \ref{K8}, \ref{K9new} $\Jdue a \land \lnot \Jdue a = 0 = \Jdue b \land \lnot \Jdue b$. So $\theta$ is a congruence of $\A$. On the other hand, it is easy to see that $D(\A)$ contains $0$ and is closed w.r.t. $J_2$. To show that it is closed w.r.t. joins, we use Item (2). If $a = x \land \lnot x, b = y \land \lnot y$, then, using a $J_2$-free regular identity valid in all Boolean algebras, $a \lor b = (x \land \lnot x) \lor (y \land \lnot y) = (x \lor y) \land \lnot (x \lor y)$. Letting $\lnot a = a$ for all $a \in D(\A)$, the resulting algebra is a semilattice which is isomorphic to $\A/\theta$ via the map $\phi(a/\theta) = a \land \lnot a$.
\end{proof}

Let $\class{V}$ the subvariety of $V(\class{BCA})$ axiomatised relative to $V(\class{BCA})$ by the single identity $J_2 \lnot x \approx \lnot J_2 x$. Clearly, $\class{BA}, \class{SL} \subseteq \class{V}$. Also, $\WK^e \notin \class{V}$. We will show that $\class{V}$ coincides with the varietal join $\class{BA} \lor \class{SL}$.

\begin{theorem}\label{marciume}
    Every $\A \in \class{V} $ is embeddable into the direct product of some $\B \in \class{BA}$ and some $\C \in \class{SL}$.
\end{theorem}

\begin{proof}
    We will embed $\A$ into $\mathbf{O}(\A) \times \mathbf{D}(\A)$, whence the result follows from Lemmas \ref{terrapieno} and \ref{calzettoni}.(3)-(4). Let $a \in A$. We define $\varphi(a) := \langle J_2a,a \land \lnot a \rangle$. By Lemma \ref{calzettoni}.(1)-(2), $\varphi$ is well-defined. $\varphi$ is also injective. Indeed, let $\varphi(a) = \varphi(b)$. Then $J_2a=J_2b$ and $ a \land \lnot a =  b \land \lnot b$, whence by Lemma \ref{lemma: aritmetica K}.(2) $a = J_2a \lor (a \land \lnot a) = J_2b \lor (b \land \lnot b) = b$.

    We show that $\varphi$ preserves the operations. Let $\B := \mathbf{O}(\A), \C := \mathbf{D}(\A)$ and $a,b \in A$.
    \begin{itemize}
        \item $\varphi(1^{\A}) = \langle J_21, 1 \land \lnot 1\rangle = \langle 1^{\B}, 1^{\C}\rangle = 1^{\B \times \C}$.
        \item $\varphi(\lnot^{\A}a) = \langle J_2 \lnot a, \lnot a \land \lnot \lnot a \rangle = \langle \lnot^{\B} J_2 a, a \land \lnot a \rangle = \lnot^{\B \times \C} \langle J_2 a, a \land \lnot a \rangle = \lnot^{\B \times \C} \varphi (a)$.
        \item $\varphi(a \land^{\A}b) = \langle J_2 (a \land b), (a \land b) \land \lnot (a \land b) \rangle = \langle J_2 a \land^{\B} J_2 b, (a \land \lnot a) \land^{\C} (b \land \lnot b) \rangle = \langle J_2 a, a \land \lnot a \rangle \land^{\B \times \C} \langle J_2 b, b \land \lnot b \rangle = \varphi(a) \land^{\B \times \C} \varphi(b)$.
        \item $\varphi(J_2^{\A}a) = \langle J_2J_2a, J_2a \land \lnot J_2a \rangle = \langle J_2^{\B}J_2a, J_2^{\C}(a \land \lnot a) \rangle = J_2^{\B \times \C} \langle J_2 a, a \land \lnot a \rangle = J_2^{\B \times \C} \varphi(a).$
    \end{itemize}
    Observe that we have crucially utilised the identity $J_2 \lnot x \approx \lnot J_2 x$, as well as some $J_2$-free regular identities valid in all Boolean algebras.
\end{proof}

\begin{corollary}
    The variety $\class{V}$ coincides with the varietal join $\class{BA} \lor \class{SL}$ in the lattice of subvarieties of $V(\class{BCA})$.
\end{corollary}

\begin{proof}
    Using Corollary \ref{destrezza} and Theorem \ref{marciume}, together with the fact that all the identities axiomatising $\class{V}$ are valid both in Boolean algebras and in semilattices, we obtain the following chain of inclusions:
    \[
    \class{V} \subseteq ISP(\class{BA} \cup \class{SL}) \subseteq \class{BA} \lor \class{SL} = \class{BA} \times \class{SL} \subseteq \class{V}.
    \]
\end{proof}

\section*{Acknowledgments}

We gratefully acknowledge the support of the following funding sources:
\begin{itemize}
    \item the University of Cagliari, through the StartUp project ``GraphNet' (CUP: F25F21002720001);
    \item the European Union (Horizon Europe Research and Innovation Programme), through the MSCA-RISE action PLEXUS (Grant Agreement no 101086295);
    \item the Italian Ministry of Education, University and Research, through the PRIN 2022 project DeKLA (``Developing Kleene Logics and their Applications'', project code: 2022SM4XC8), the PRIN 2022 project “The variety of grounding” (project code: 2022NTCHYF), and the PRIN Pnrr project ``Quantum Models for Logic, Computation and Natural Processes (Qm4Np)'' (project code: P2022A52CR);
    \item the Fondazione di Sardegna, through the project MAPS (grant number F73C23001550007);
    \item the INDAM GNSAGA (Gruppo Nazionale per le Strutture Algebriche, Geometriche e loro Applicazioni);
    \item the CARIPARO Foundation, through the excellence project (2020-2024): “Polarization of irrational collective beliefs in post-truth societies”.
\end{itemize}
We thank Antonio Ledda, Gavin St. John, Sara Ugolini, Nicol\'o Zamperlin for discussing with us the subject matter of this paper.

\noindent


\end{document}